\documentclass[pdflatex,sn-mathphys]{sn-jnl}% Math and Physical Sciences Reference Style
\usepackage[all]{xy}

\DeclareMathOperator{\rad}{\text{rad}}
\DeclareMathOperator{\Hom}{\text{Hom}}

\DeclareMathOperator{\md}{\operatorname{mod}}

\DeclareMathOperator{\Aa}{\mathcal{A}}
\DeclareMathOperator{\dimk}{\operatorname{dim}_k}

\DeclareMathOperator{\End}{\operatorname{End}}

\newcommand{\lra}{\longrightarrow}

\newcommand{\benu}{\begin{enumerate}}
\newcommand{\enu}{\end{enumerate}}

\newcommand{\Co}{{\mathcal C}}
\newcommand{\I}{\mathcal{I}}
\newcommand{\dCo}{{\mathcal C}^*}

\jyear{2021}%

\theoremstyle{thmstyleone}%
\newtheorem{theorem}{Theorem}%  meant for continuous numbers
\newtheorem{proposition}[theorem]{Proposition}% 

\theoremstyle{thmstyletwo}%
\newtheorem{example}{Example}%
\newtheorem{remark}{Remark}%

\theoremstyle{thmstylethree}%
\newtheorem{definition}{Definition}%

\begin{document}

\title[Representations]{Representations of Generalized Bound Path Algebras}


\author[1]{\fnm{Viktor} \sur{Chust}}\email{viktorch@ime.usp.br}
%\equalcont{ORCID:0000-0003-4931-4222}

\author[2]{\fnm{Flávio} \sur{U. Coelho}}\email{fucoelho@ime.usp.br}
%\equalcont{This author contributed equally to this work.}


\affil[1]{\orgdiv{ORCID: 0000-0003-4931-4222. Institute of Mathematics and Statistics}, \orgname{University of São Paulo}, \orgaddress{\street{R. do Matão, 1010}, \city{São Paulo}, \postcode{05508-090}, \state{São Paulo}, \country{Brazil}}}

\affil[2]{\orgdiv{ORCID: 0000-0002-1292-621X. Institute of Mathematics and Statistics}, \orgname{University of São Paulo}, \orgaddress{\street{R. do Matão, 1010}, \city{São Paulo}, \postcode{05508-090}, \state{São Paulo}, \country{Brazil}}}

%\affil[2]{\orgdiv{Department}, \orgname{Organization}, \orgaddress{\street{Street}, \city{City}, \postcode{10587}, \state{State}, \country{Country}}}

%\affil[3]{\orgdiv{Department}, \orgname{Organization}, \orgaddress{\street{Street}, \city{City}, \postcode{610101}, \state{State}, \country{Country}}}

\abstract{
The concept of generalized path algebras was introduced in (Coelho, Liu, 2000). Roughly speaking, these algebras are constructed in a similar way to that of the path algebras over a quiver, the difference being that we assign an algebra to each vertex of the quiver and consider paths intercalated with elements from these algebras. Then we use concatenation of paths together with the algebra structure in each vertex to define multiplication. The representations of a generalized path algebra were described in one of the main results of (Ibáñez Cobos et al., 2008), in terms of the representations of the algebras used in its construction. In this article, we continue our investigation started in (Chust, Coelho, 2021) and extend the result mentioned above to describe the representations of the generalized bound path algebras, which are a quotient of generalized path algebras by an ideal generated by relations. In particular, the representations oassociated with the projective and injective modules are described.
}
%\abstract{
%In \cite{ICNLP}, Cobos-Navarro-Pe\~na describe the representations of generalized path algebras in the sense introduced by Coelho-Liu \cite{CLiu}. Here, following the investigation started in \cite{CC1}, we extend those results to the setting of  generalized bound path algebras, which are  quotients of generalized path algebras by ideals generated by relations. In particular, the representations of projective and injective modules are given. }



\maketitle

\section{Introduction}\label{sec1}

It is a well-established fact that any finite dimensional basic algebra $A$ over an algebraically closed field $k$ can be seen as the quotient of a path algebra, that is, $ A \cong kQ/I$, where $Q$ is a quiver and $I$ is an admissible ideal of $kQ$ (see for instance \cite{AC, ARS}).  In \cite{CLiu}, Coelho and Liu studied a generalization of such construction. There, it is assigned an algebra to each vertex of a given quiver $Q$ instead of just assigning the base field. The multiplication in such a generalization  will be given not only by the concatenation of paths on the quiver but also by those of the algebras associated with the vertices. 

More specifically, let $\Gamma$ denote a quiver and  $\Aa=\{A_i : i \in \Gamma_0\}$ denote a family of basic algebras of finite dimension over an algebraically closed field $k$ indexed by the set $\Gamma_0$ of the vertices of $\Gamma$. Consider also a set of relations $I$ on the paths of $\Gamma$. In \cite{CC1}, to such a data we assigned a generalized bound path algebra $\Lambda = k(\Gamma,\Aa,I)$ with a natural multiplication (see preliminaries for details).

In \cite{CLiu}, where it is considered the particular case when $I = 0$,  the main interest was more of ring-theoretic nature, but clearly, such a construction can be also very useful from the point of view of Representation Theory. In  \cite{CC1}, we start our work in this direction for the general case. Observe that  any algebra $A$ can be naturally realized as a generalized bound path algebra in two ways. Firstly, the well-known description as the usual quotient of a path algebra. But also, $A$ can be seen by using a quiver with a sole vertex and no arrows and the algebra itself assigned to it. Since for most algebras, these are the only possibilities, one can wonder for which algebras it is possible to describe them as  generalized bound path algebras in a different way from these two above (we call it a {\bf non-trivial simplification of $A$}). Such a description could be useful once one aims to look at properties of a given algebra from those of smaller ones. We deal with this problem in \cite{CC1}. 

Here, following the same strategies of our previous work, the focus will be on the representations of a generalized bound path algebra. When $I= 0$, this has been considered in 
\cite{ICNLP} and we shall generalize their results here (Theorem~\ref{th:rep_icnlp}). Descriptions of the representation of the projective, injective and simple modules are also given.  

This paper is organized as follows. Section 2 below is devoted to the preliminaries needed along the paper. In Section 3 we prove the above mentioned theorem which describes the representations of a given generalized bound path algebra. After establishing useful ideas in Section 4, Section 5 is devoted to the description of the projective, injective and simple modules. 

In a forthcoming paper \cite{CC3}, we shall look at the homological relations between the algebras $A_i$ and the whole algebra.

\section{Preliminaries} 

We shall here recall some basic notions and establish some notations needed along the paper. We indicate the books \cite{AC, ARS} where details on Representation Theory can be found. For an algebra, we shall mean an associative and unitary basic algebra of finite dimension over an algebraically closed field $k$. Unless otherwise stated, the modules considered here are right modules.

%%%%%%%%%%%%%%%%%%%%
\subsection{Quivers and path algebras}

\label{subsec:quivers}

A {\bf  quiver} $Q$  is given by $(Q_0, Q_1, s,e)$ where $Q_0$ is the set of {\bf vertices}, $Q_1$ is the set of {\bf arrows} and $s,e \colon Q_1 \lra Q_0$ are functions which indicate, for each arrow $\alpha \in Q_1$, the {\bf starting vertex} $s(\alpha)\in Q_0$ of $\alpha$ and the {\bf ending vertex} $e(\alpha) \in Q_0$ of $\alpha$. Naturally, given a quiver $Q$ one can assign a {\bf path algebra} $kQ$ with a $k$-basis given by all paths of $Q$ and multiplication on that basis defined by concatenation. Even when $Q$ is finite (that is, when $Q_0, Q_1$ are finite sets), the corresponding algebra could not be finite dimensional. However, a well-known result established by Gabriel states that given an algebra $A$, there exists a finite quiver $Q$ and a set of relations on the paths of $Q$ which generates an admissible ideal $I$ such that $A \cong kQ/I$ (see \cite{AC} for details). 

Along this paper we will assume that the quivers are finite. 

%%%%%%%%%%%%%%%%%%
\subsection{Generalized path algebras} 

We shall now recall the definition of a generalized path algebra given in \cite{CLiu}.

Let $\Gamma= (\Gamma_0,\Gamma_1,s,e)$ be a quiver and $\mathcal{A}=(A_i)_{i \in \Gamma_0}$ be a family of algebras, indexed by $\Gamma_0$.   An \textbf{$\mathcal{A}$-path of length $n$} over $\Gamma$ is defined as follows. If $n = 0$, it is just an element of $\bigcup_{i \in \Gamma_0} A_i$, and, if $n>0$, it is a  sequence of the form
$$a_1 \beta_1 a_2 \ldots a_n \beta_n a_{n+1}$$
where $\beta_1 \ldots \beta_n$ is an ordinary path over $\Gamma$, $a_i \in A_{s(\beta_i)}$ if $i \leq n$, and $a_{n+1} \in A_{e(\beta_n)}$.  Denote by $k[\Gamma,\mathcal{A}]$ the $k$-vector space spanned by all $\mathcal{A}$-paths over $\Gamma$. We shall give it a structure of algebra as follows. 

Firstly, consider the quotient vector space $k(\Gamma,\mathcal{A})=k[\Gamma,\mathcal{A}]/M$, where $M$ is the subspace generated by all elements of the form
$$(a_1 \beta_1 \ldots \beta_{j-1}(a^1_j+ \ldots+a^m_j)\beta_j a_{j+1} \ldots \beta_n a_{n+1}) - \sum_{l=1}^m (a_1 \beta_1 \ldots \beta_{j-1} a_j^l \beta_j \ldots \beta_n a_{n+1})$$
or, for $\lambda \in k$,
$$(a_1 \beta_1 \ldots \beta_{j-1}( \lambda a_j) \beta_j a_{j+1} \ldots \beta_n a_{n+1})- \lambda\cdot (a_1 \beta_1 \ldots \beta_{j-1} a_j \beta_j a_{j+1} \ldots \beta_n a_{n+1})$$

Now, consider the multiplication  in $k(\Gamma,\mathcal{A})$  induced by the multiplications of the $A_i$'s and by composition of paths. Namely, it is defined by linearity and the following rule:
$$(a_1 \beta_1 \ldots \beta_n a_{n+1})(b_1 \gamma_1 \ldots \gamma_m b_{m+1}) = a_1 \beta_1 \ldots \beta_n (a_{n+1} b_1) \gamma_1 \ldots \gamma_m b_{m+1}$$
if $e(\beta_n) = s(\gamma_1)$, and 
$$(a_1 \beta_1 \ldots \beta_n a_{n+1})(b_1 \gamma_1 \ldots \gamma_m b_{m+1}) = 0 $$
otherwise.

With this multiplication, we call $k(\Gamma,\mathcal{A})$  the   \textbf{generalized path algebra} over $\Gamma$ and $\mathcal{A}$.

\begin{remark}
It should be easy to see that the ordinary path algebras are a particular case of generalized path algebras, simply by taking $A_i = k$ for every $i \in \Gamma_0$.
\end{remark}

Note that the generalized path algebra $k(\Gamma,\Aa)$ is an associative algebra. And since we are assuming the quivers to be finite, it also has an identity element, which is equal to $\sum_{i \in \Gamma_0} 1_{A_i}$. Finally, it is easy to observe that $k(\Gamma,\Aa)$ is finite-dimensional over $k$ if and only if so are the algebras $A_i$ and if $\Gamma$ is acyclic. 

\begin{remark}
\label{obs:gpa_tensor_algebra}

As observed in \cite{CLiu}, if $k(\Gamma,\Aa)$ is a generalized path algebra as defined above, then it is a tensor algebra: if $A_{\Aa} = \prod_{i \in \Gamma_0} A_i$ is the product of the algebras in $\Aa$, then there is an $(A_{\Aa}-A_{\Aa})$-bimodule $M_{\Aa}$ such that $k(\Gamma,\Aa) \cong T(A_{\Aa},M_{\Aa})$.
\end{remark}

\subsection{Generalized bound path algebras (gbp-algebras)}

Following \cite{CC1}, we shall extend the definition of generalized path algebras to allow them to have relations.  In doing so, these algebras will be called \textbf{generalized bound path algebras} or \textbf{gbp-algebras} to abbreviate. As observed in \cite{CC1}, the idea of taking the quotient of a generalized path algebra by an ideal of relations has already been studied by Li Fang (see \cite{FLi1} for example). However, the concept dealt with in \cite{CC1} and here is slightly different, since in order to prove the results below, we consider an ideal of relations which is in general bigger roughly speaking.

Observe that if $A_i \in \Aa$, then, as explained in Subsection~\ref{subsec:quivers}, there is a quiver $\Sigma_i$ such that $A_i \cong k\Sigma_i/ \Omega_i$ where $\Omega_i$ is an admissible ideal of $k\Sigma_i$. Let now $I$ be a finite set of relations over $\Gamma$ which generates an admissible ideal in $k\Gamma$.  Consider the ideal $(\Aa(I))$ generated by the following subset of  $k(\Gamma,\Aa)$:
\begin{align*}
\Aa(I)&= \left\{ \sum_{i = 1}^t \lambda_i  \beta_{i1} \overline{\gamma_{i1}} \beta_{i2} \ldots \overline{\gamma_{i(m_i-1)}} \beta_{im_i} : \right.\\
&\left. \sum_{i = 1}^t \lambda_i \beta_{i1} \ldots \beta_{im_i} \text{  is a relation in } I 
\text{ and }\gamma_{ij}\text{ is a path in }\Sigma_{e(\beta_{ij})} \right\}
\end{align*}

The quotient $\frac{k(\Gamma,\mathcal{A})}{(\Aa(I))}$ is said to be a \textbf{generalized bound path algebra (gbp-algebra)}. To simplify the notation, we may also write $\frac{k(\Gamma,\mathcal{A})}{(\Aa(I))}=k(\Gamma,\Aa,I)$. When the context is clear, we may denote the set $\Aa(I)$ simply by $I$.

%%%%%%%%%%%%5
\subsection{Notations}

We are going to use the following notation in this article: $\Gamma$ will always be an acyclic quiver, $\Aa=\{A_i : i \in \Gamma_0\}$ will denote a family of basic algebras of finite dimension over an algebraically closed field $k$, and $I$ will be a set of relations in $\Gamma$ generating an admissible ideal in the path algebra $k\Gamma$. We will also denote by $\Lambda = k(\Gamma,\Aa,I)$ the generalized bound path algebra (gbp-algebra) obtained from these objects. Also, $A_{\Aa}$ will denote the product algebra $\prod_{i \in \Gamma_0} A_i$.  For the purpose of simplifying notation, we are also going to denote the identity element of the algebras $A_i$ by $1_i$ instead of $1_{A_i}$.

%%%%%%%%%%%%%%%%%%% 
\section{Representations}

The aim of this section is to prove Theorem~\ref{th:rep_icnlp} below, which is an extension of Theorem 2.4 from \cite{ICNLP}. As already mentioned above, this result will be of key importance here, and sometimes we will be using it without further clarification.

Based on \cite{ICNLP}, we start by defining what are generalized representations. However, before this we need to do a remark about the notation used here:

\begin{remark}
Generally speaking, if $A$ is an algebra and $M$ is a vector space, an action of $A$ over $M$ which turns $M$ into an $A$-module is equivalent to a homomorphism of algebras $\phi: A \rightarrow \End_k M$. (This correspondence is given by $\phi(a)(m) = m.a$ for all $a \in A$ and $m \in M$). That way, if we understand this correpondence as being canonical, then, at least in the concepts to be treated below, an element $a$ of $A$ could denote either the element itself or $\phi(a)$, which is the endomorphism given by right translation through $a$: $m \mapsto m.a$ for all $m \in M$. This shall be done in order to simplify the notations.
\end{remark}

\begin{definition} Let $\Lambda = k(\Gamma,\Aa,I)$ be a generalized bound path algebra. 

(a) A {\bf representation} of $\Lambda$ is given by 
$ ((M_i)_{i \in \Gamma_0},(M_{\alpha})_{\alpha \in \Gamma_1})$
where
\begin{enumerate}

\item[(i)]  for every $i \in \Gamma_0$, $M_i$ is an $A_i$-module;
\item[(ii)] for every arrow $\alpha \in \Gamma_1$, $M_{\alpha}: M_{s(\alpha)} \rightarrow M_{e(\alpha)}$ is a $k$-linear transformation. 
\item[(iii)] it satisfies any relation $\gamma$ of $I$. That is, if $\gamma = \sum_{i = 1}^t \lambda_t \alpha_{i1} \alpha_{i2} \ldots \alpha_{in_1}$ is a relation in $I$ with $\lambda_i \in k$ and $\alpha_{ij} \in \Gamma_1$, then 
$$\sum_{i = 1}^t \lambda_t M_{\alpha_{in_i}} \circ \overline{\gamma_{in_i}} \circ \ldots \circ M_{\alpha_{i2}} \circ \overline{\gamma_{i2}} \circ M_{\alpha_{i1}} = 0$$
for every choice of paths $\gamma_{ij}$ over $\Sigma_{s(\alpha_{ij})}$, with $1 \leq i \leq t$, $2 \leq j \leq n_i$.
\end{enumerate}

(b) We say that a representation $ ((M_i)_{i \in \Gamma_0},(M_{\alpha})_{\alpha \in \Gamma_1})$ of $\Lambda$ is {\bf finitely generated} if each of the $A_i$-modules $M_i $ is finitely generated. 

(c) Let $M = ((M_i)_{i \in \Gamma_0},(M_{\alpha})_{\alpha \in \Gamma_1})$ and $N = ((N_i)_{i \in \Gamma_0},(N_{\alpha})_{\alpha \in \Gamma_1})$ be representations of  $\Lambda$. A \textbf{morphism of representations} $f: M \rightarrow N$ is given by a tuple $f = (f_i)_{i \in \Gamma_0}$, such that, for every $i \in \Gamma_0$, $f_i: M_i \rightarrow N_i$ is a morphism of $A_i$-modules; and such that, for every arrow $\alpha: i \rightarrow j \in \Gamma_1$, it holds that $f_j M_{\alpha} = N_{\alpha} f_i$, that is, the following diagram comutes:
$$ \xymatrix{M_i \ar[r]^{M_{\alpha}} \ar[d]_{f_i} & M_j \ar[d]^{f_j}\\
N_i \ar[r]_{N_{\alpha}} & N_j} $$
\end{definition}

We shall denote by Rep$_k(\Gamma,\Aa,I)$ (or rep$_k(\Gamma,\Aa,I)$, respectively) the category of the representations (or finitely generated representations) of the algebra $k(\Gamma,\Aa,I)$.

The next step will be to establish the promised equivalence between $k(\Gamma,\Aa,I)$-representations and $\Lambda$-modules, thus generalizing the well-known result of Gabriel for representions and also Theorem 2.4 from \cite{ICNLP}, where the equivalence was established only in the case $I = \emptyset$. The construction of the functors $F$ and $G$ is essentially the same of the original proof, but, for completeness, we will repeat it here. 

\begin{theorem}[compare with \cite{ICNLP},2.4]
\label{th:rep_icnlp}
There is a a $k$-linear equivalence 
$$F: \operatorname{Rep}_k(\Gamma,\Aa,I) \rightarrow \operatorname{Mod} k(\Gamma,\Aa,I)$$
which restricts to an equivalence 
$$F: \operatorname{rep}_k(\Gamma,\Aa,I) \rightarrow \operatorname{mod} k(\Gamma,\Aa,I)$$
\end{theorem}

\begin{proof}
For a given representation $M = ((M_i)_{i \in \Gamma_0},(M_{\alpha})_{\alpha \in \Gamma_1})$ in  $\operatorname{Rep}_k(\Gamma,\Aa,I)$, define 
$$F(M) = \bigoplus_{i \in \Gamma_0} M_i$$
which will be an object in $\operatorname{Mod} k(\Gamma,\Aa,I)$.\\
We have to define the action of $\Lambda$ over $F(M)$ in such a way that $F(M)$ is indeed an object in $\operatorname{Mod} k(\Gamma,\Aa,I)$. This is equivalent to constructing a homomorphism of algebras $\Phi: \Lambda \rightarrow \End F(M)$. The idea is to use the universal property of tensor algebras (see \cite{ICNLP}, Lemma 2.1).  Let $A_{\Aa}$ and $M_{\Aa}$ be as in Remark~\ref{obs:gpa_tensor_algebra}. \\
First we define a homomorphism of algebras 
$$\phi_0: A_{\Aa} \rightarrow \End_k F(M)$$
given by
$$ \phi_0(a_i)((x_l)_{l \in \Gamma_0}) = (\delta_{li} x_i a_i)_{l \in \Gamma_0}$$
for all $i \in \Gamma_0,$ for all $ a_i \in A_i$ and all $ (x_l)_{l \in \Gamma_0} \in F(M)$, 
where $\delta_{li}$ is a Kronecker's delta. We also define a morphism of $(A_{\Aa} - A_{\Aa})$-bimodules
$$\phi_1: M_{\Aa} \rightarrow \End_k F(M)$$
as follows: for every $\Aa$-path of length 1 $a_i \alpha a_j$, where $\alpha: i \rightarrow j$ is an arrow of $\Gamma$, $a_i \in A_i$, $a_j \in A_j$, and for every tuple $(x_l)_{l \in \Gamma_0} \in F(M)$, define
$$\phi_1(a_i \alpha a_j)((x_l)_{l \in \Gamma_0}) = (\delta_{lj} M_{\alpha} (x_ia_i) a_j)_{l \in \Gamma_0}$$
Now, since $k(\Gamma,\Aa) = T(A_{\Aa},M_{\Aa})$, by the universal property of tensor algebras (\cite{ICNLP},Lemma 2.1), there is a homomorphism of algebras
$$\phi: k(\Gamma,\Aa) \rightarrow \End_k F(M)$$
uniquely determined by the property that $\phi \vert_{A_{\Aa}} = \phi_0$ and $\phi \vert_{M_{\Aa}} = \phi_1$. This shows that $F(M)$ is a $k(\Gamma,\Aa)$-module. In order to show that $F(M)$ is a module over $\Lambda = k(\Gamma,\Aa,I)$, is suffices to show that $\phi(I)=0$, because then, due to the Homomorphism Theorem, $\phi$ induces a homomorphism of algebras $\Phi: k(\Gamma,\Aa)/I \rightarrow \End_k F(M)$. \\
Therefore let us verify that $\phi(I)=0$. Let $\rho = \sum_{r=1}^t \lambda_r \alpha_{r1} \ldots \alpha_{rn_r}$ be a relation in $I$, where $\lambda_r\in k$ and the sequences $\alpha_{r1} \ldots \alpha_{rn_r}$ are paths over $\Gamma$ that start and end at the same vertex. And let, for every $1 \leq r \leq t$ and $2 \leq j \leq n_r$, $\gamma_{rj}$ be a path over $\Sigma_{s(\alpha_{rj})}$. Then:
\begin{align*}
\phi&(\sum_{r = 1}^t \lambda_r \alpha_{r1} \overline{\gamma_{r2}} \alpha_{r2} \ldots \overline{\gamma_{rn_r}} \alpha_{rn_r}) \\
&=\sum_{r = 1}^t \lambda_r \phi(\alpha_{r1} \overline{\gamma_{r2}} \alpha_{r2} \ldots \overline{\gamma_{rn_r}} \alpha_{rn_r}) \\
& = \sum_{r = 1}^t \lambda_r \imath_{e(\alpha_{rn_r})} \circ M_{\alpha_{rn_r}} \circ \overline{\gamma_{rn_r}} \circ \ldots M_{\alpha_{r2}} \circ \overline{\gamma_{r2}} \circ M_{\alpha_{r1}} \circ \pi_{s(\alpha_{r1})} \\
& = \imath_{e(\alpha_{1n_1})} \circ \left(\sum_{r = 1}^t \lambda_r M_{\alpha_{rn_r}} \circ \overline{\gamma_{rn_r}} \circ \ldots M_{\alpha_{r2}} \circ \overline{\gamma_{r2}} \circ M_{\alpha_{r1}}\right) \circ \pi_{s(\alpha_{11})} \\
&=0
\end{align*}
where $\imath$ and $\pi$ denote respectively canonical inclusions and projections, and the last equality above holds because $M$ satisfies $\rho$. We need to see how $F$ acts on morphisms. \\ 
Let $f = (f_i)_{i \in \Gamma_0}: M \rightarrow N$ be a morphism of representations, where $M = ((M_i)_{i \in \Gamma_0},(M_{\alpha})_{\alpha \in \Gamma_1})$ and $N = ((N_i)_{i \in \Gamma_0},(N_{\alpha})_{\alpha \in \Gamma_1})$ are representations satisfying $I$. Then each $f_i: M_i \rightarrow N_i$ is a morphism of $A_i$-modules, and thus we may define a linear map
$$F(f) : F(M) = \bigoplus_{i \in \Gamma_0} M_i \rightarrow F(N) = \bigoplus_{j \in \Gamma_0} N_i$$
by establishing that the $(i,j)$-th coordinate of $F(f)$ is $\delta_{ij} f_i$. It can be shown that $F(f)$ is a morphism of $\Lambda$-modules and that $F$ defined as such is indeed a functor. \\ 
Now we will define that which will be the quasi-inverse functor of $F$:
$$G: \operatorname{Mod} k(\Gamma,\Aa) \rightarrow \operatorname{Rep}_k(\Gamma,\Aa)$$
Let $M$ be a module over $\Lambda$. We need to define a $k(\Gamma,\Aa)$-representation $G(M) = ((M_i)_{i \in \Gamma_0},(\phi_{\alpha})_{\alpha \in \Gamma_1})$ which satisfies $I$.  
\begin{itemize}
\item For each $i \in \Gamma_0$, $M_i$ is defined by $M_i \doteq M\cdot 1_i$ which is clearly an $A_i$-module. 
\item  For each arrow $\alpha: i \rightarrow j \in \Gamma_1$, define the $k$-linear  map $M_{\alpha}:M_i \rightarrow M_j$ given by $\phi_{\alpha}(m) = m\cdot \alpha$. 
\end{itemize}
To show that $G(M)$ thus defined satisfies $I$, let 
$\rho = \sum_{r=1}^t \lambda_r \alpha_{r1} \ldots \alpha_{rn_r}$ be a relation in $I$, where $\lambda_r\in k$ and the sequences $\alpha_{r1} \ldots \alpha_{rn_r}$ are paths over $\Gamma$ that start and end at the same vertex. Also let, for each $1 \leq r \leq t$ and $2 \leq j \leq n_r$, $\gamma_{rj}$ be a path over $\Sigma_{s(\alpha_{rj})}$. Then, for $m \in M_{s(\alpha_{r1})}$,
\begin{align*}
&\left( \sum_{r=1}^t  \lambda_r  M_{\alpha_{rn_r}} \circ \overline{\gamma_{rn_r}} \circ \ldots M_{\alpha_{r2}} \circ \overline{\gamma_{\alpha_{r2}}} \circ M_{\alpha_{r1}} \right) (m) \\
&=\left(\sum_{r=1}^t \lambda_r M_{\alpha_{rn_r}} \circ \overline{\gamma_{rn_r}} \circ \ldots M_{\alpha_{r2}} \circ \overline{\gamma_{r2}} \right)(m \alpha_{r1}) \\
&= \left(\sum_{r=1}^t \lambda_r M_{\alpha_{rn_r}} \circ \overline{\gamma_{rn_r}} \circ \ldots M_{\alpha_{r2}}\right) (m \alpha_{r1} \overline{\gamma_{r2}}) \\
&= \ldots = \sum_{r=1}^t \lambda_r m \alpha_{r1} \overline{\gamma_{r2}} \ldots \overline{\gamma_{rn_r}} \alpha_{rn_r} \\
&= m \left(\sum_{r=1}^t \lambda_r \alpha_{r1} \overline{\gamma_{r2}} \ldots \overline{\gamma_{rn_r}} \alpha_{rn_r} \right) \\
&= 0
\end{align*}
The last equality above holds because the expression that multiplies $m$ is equal to $0$ in $\Lambda$. We have thus shown that $G(M)$ is an object in $\operatorname{Rep}_k(\Gamma,\Aa,I)$. \\ 
Let $g: M \rightarrow N$ be a morphism in $\operatorname{Mod} \Lambda$. We will define its image under $G$:
\begin{align*}
&G(g) = (G(g)_i)_{i \in \Gamma_0} \\
&G(g)_i : M_i \rightarrow N_i\text{, } G(g)_i \doteq g \vert_{M_i}
\end{align*}
It is immediately verified that $G(g)_i$ is well-defined and is a morphism of $A_i$-modules for every $i \in \Gamma_0$. Let us show that $G(g)$ is a morphism of representations. Let $\alpha: i \rightarrow j$ be an arrow in $\Gamma$. Then, for every $m \in M$, $G(g)_j \circ M_{\alpha} (m\cdot 1_i) = G(g)_j (m\alpha) = g(m\alpha) = g(m)\alpha = G(g)_i(m\cdot 1_i)\alpha = N_{\alpha} \circ G(g)_i (m\cdot 1_i)$. Therefore $G(g)_j \circ M_{\alpha} = N_{\alpha} \circ G(g)_i$, that means to say that the following diagram comutes:
\begin{displaymath}
\xymatrix{
M_i \ar[r]^{M_{\alpha}} \ar[d]_{G(g)_i} & M_j \ar[d]^{G(g)_j} \\
N_i \ar[r]_{N_{\alpha}} &N_j
}
\end{displaymath}
Therefore $G(g)$ is a morphism of representations. It is straightforward to show $G$ defined this way is a functor. It is also directly verified that:

\begin{itemize}
\item $F$ and $G$ are quasi-inverse functors and are therefore equivalences.
\item $F$ maps finitely generated representations to finitely generated modules, while $G$ does the opposite. Thus the restrictions of these functors to these subcategories are still quasi-inverse equivalences.
\end{itemize}
\end{proof}

\begin{example}

In this example we will illustrate Theorem~\ref{th:rep_icnlp} above. Let $A$ be the path algebra given by the quiver 

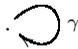
\begin{displaymath}
\xymatrix{. \ar@(ur,dr)[]^{\gamma}}
\end{displaymath}
bound by $\gamma^n = 0$, where $n>1$. Then consider the  gbp-algebra $\Lambda = k(\Gamma,\Aa,I)$, where $\Gamma$ is the quiver below:

\begin{displaymath}
\xymatrix{1 \ar[rr]^{\alpha} &&2 \ar[rr]^{\beta}&&3}
\end{displaymath}
and where $\Aa = \{A_1,A_2,A_3\}$, with $A_1 = A_3 = k$, $A_2 = A$, and $I = (\alpha \beta)$. More simply, $\Lambda$ is the gbp-algebra given by

\begin{displaymath}
\xymatrix{k \ar[rr]^{\alpha} &&A \ar[rr]^{\beta}&&k}
\end{displaymath}
bound by $\alpha \beta = 0$. Using the proof of Theorem~\ref{th:rep_icnlp}, we are going to calculate the representation associated with the projective $\Lambda$-module $P = 1_{A_1}\cdot \Lambda$. \\
We have that $P_1 = P.1_{A_1} = 1_{A_1}.\Lambda.1_{A_1} = (1_{A_1})$ is the $k$-vector space spanned by $1_{A_1}$. Moreover, $P_2 = P.1_{A_2} = 1_{A_1}.\Lambda.1_{A_2} = (\alpha, \alpha \overline{\gamma}, \ldots \alpha \overline{\gamma}^{n-1})$, which is a right $A$-module easily seen to be isomorphic to the regular $A$-module $A$. And also $P_3 = P.1_{A_3} = 1_{A_1}.\Lambda.1_{A_3} = 0$ since $I=(\alpha\beta)$ and thus every $\Aa$-path of the form $\alpha \overline{\gamma}^i \beta$ for $i \geq 0$ is identified with 0 in $\Lambda$.\\
Now we have that $P_{\alpha}$ is given by right multiplication by $\alpha$, so it maps the single element of the basis of $P_1$, which is $1_{A_1}$, to $1_{A_1}. \alpha = \alpha$ in $P_2$.\\
If we identify $P_2 \cong A$ and consider the $k$-basis $\{\overline{1},\overline{\gamma},\ldots, \overline{\gamma}^{n-1}\}$ for $A$, we may conclude that the representation associated with the $\Lambda$-module $P$ is the following:

\begin{displaymath}
\xymatrix{P: & k \ar[rrr]^{\left[ \begin{matrix} 1 & 0 & \ldots & 0 \end{matrix} \right]^T} &&&A \ar[rr]&&0}
\end{displaymath}

\end{example} 
\vspace{.3 cm}

Having obtained the equivalence in Theorem~\ref{th:rep_icnlp} as a tool, we are in conditions to study, over the course of the following sections, the representations associated to simple, projective and injective modules over a gbp-algebra, thus generalizing the well-known description that is done for ordinary path algebras.

\begin{remark}
From now on, we will also be additionally assuming that the modules are always finitely generated.
\end{remark}

\subsection{Opposite algebra}
\label{subsec:opp_alg_duality}

The aim of this subsection is to obtain some useful lemmas involving opposite algebras, opposite quivers and the duality functor. Again we refer to \cite{ARS} for the definition of these concepts. For a quiver $\Gamma$, denote by $\Gamma^{op}$ its opposite quiver (that is, the quiver with the same vertices of $\Gamma$ and with all its arrows reversed). For a set $I$ of relations in $\Gamma$, $I^{op}$ will denote the  set of relations in $\Gamma^{op}$ obtained through inversion of the arrows in $I$. Also, if $\Aa = \{A_i : i \in \Gamma_0\}$ is a family of algebras, denote by  $\Aa^{op} = \{A_i^{op}:i \in \Gamma_0\}$ the set where $A_i^{op}$ is the opposite algebra of $A_i$. With these notations, we have the following:

\begin{proposition}
\label{prop:opp_GPA}
If $\Lambda = k(\Gamma,\Aa, I)$ is a gbp-algebra, then $\Lambda^{op} \cong k(\Gamma^{op},\Aa^{op},I^{op})$. 
\end{proposition}
\begin{proof}
As recalled in the preliminaries, the generalized path algebra $k(\Gamma,\Aa)$ is a quotient of a vector space denoted as $k[\Gamma,\Aa]$ by a subspace generated by linearity relations. Let us then use the following auxiliar notation: $k(\Gamma,\Aa) \doteq k[\Gamma,\Aa]/\sim$. In order to avoid confusion, let us also denote the equivalence class (relatively to $\sim$) of an $\Aa$-path $x$ by $[x]$. With these notations we can define a $k$-linear map
$$\overline{\phi}:k[\Gamma,\Aa] \rightarrow k(\Gamma^{op},\Aa^{op})$$
by defining  it in the $k$-basis of $k[\Gamma,\Aa]$:
$$\overline{\phi}(a_0 \beta_1 a_1 \ldots a_{r-1} \beta_r a_r) \doteq [a_r \beta_r a_{r-1} \ldots a_1 \beta_1 a_0]$$
for each $\Aa$-path $a_0 \beta_1 a_1 \ldots a_{r-1} \beta_r a_r$. Then we must show that $\sim \subseteq \ker \overline{\phi}$. Indeed:
$$ \overline{\phi}(a_0 \beta_1 a_1 \ldots (a_i^1+\ldots+a_i^s) \ldots a_{r-1} \beta_r a_r - \sum_{j=1}^s a_0 \beta_1 a_1 \ldots a_i^j \ldots a_{r-1} \beta_r a_r)  = $$ 
$$ = \overline{\phi}(a_0 \beta_1 a_1 \ldots (a_i^1+\ldots+a_i^s) \ldots a_{r-1} \beta_r a_r) - \sum_{j=1}^s \overline{\phi}(a_0 \beta_1 a_1 \ldots a_i^j \ldots a_{r-1} \beta_r a_r) = $$  
$$ = [a_r \beta_r a_{r-1} \ldots (a_i^1+\ldots+a_i^s) \ldots a_1 \beta_1 a_0]- \sum_{j=1}^s [a_r \beta_r a_{r-1} \ldots a_i^j \ldots a_1 \beta_1 a_0] = 0 $$
and, for  $\lambda \in k$,
$$ \overline{\phi} (a_0 \beta_1 a_1 \ldots \lambda a_i \ldots a_{r-1} \beta_r a_r - \lambda(a_0 \beta_1 a_1 \ldots a_i \ldots a_{r-1} \beta_r a_r))  = $$
$$ = \overline{\phi}(a_0 \beta_1 a_1 \ldots \lambda a_i \ldots a_{r-1} \beta_r a_r) - \lambda \overline{\phi}(a_0 \beta_1 a_1 \ldots a_i \ldots a_{r-1} \beta_r a_r) = $$ 
$$ = [a_r \beta_r a_{r-1} \ldots \lambda a_i \ldots a_1 \beta_1 a_0] - \lambda [a_r \beta_r a_{r-1} \ldots a_i \ldots a_1 \beta_1 a_0] = 0 $$
We have just shown that there is a $k$-linear map
$$\phi : k(\Gamma,\Aa) \rightarrow k(\Gamma^{op},\Aa^{op})$$
that satisfies
$$\phi([a_0 \beta_1 a_1 \ldots a_{r-1} \beta_r a_r]) = [a_r \beta_r a_{r-1} \ldots a_1 \beta_1 a_0]$$
It is easy to see that $\phi$ is bijective. 
To conclude the first part of the statement, it remains to show that $\phi$ is an anti-homomorphism of algebras. It is easy to see that $\phi$ preserves the identity element. We will thus show that it antipreserves multiplication. Let $a = [a_0 \beta_1 a_1 \ldots a_{r-1} \beta_r a_r]$ and $b= [b_0 \gamma_1 b_1 \ldots b_{s-1} \gamma_s b_s]$ be the classes of two $\Aa$-paths. If $e(\beta_r) \neq s(\gamma_1)$, it is straightforward to show that $\phi(ab) = 0 = \phi(b)\phi(a)$. So suppose that $e(\beta_r) = s(\gamma_1)$. In this case,
\begin{align*}
\phi(ab) &= \phi([a_0 \beta_1 a_1 \ldots a_{r-1} \beta_r a_r][b_0 \gamma_1 b_1 \ldots b_{s-1} \gamma_s b_s]) \\
&= \phi ([a_0 \beta_1 a_1 \ldots a_{r-1} \beta_r (a_r.b_0) \gamma_1 b_1 \ldots b_{s-1} \gamma_s b_s]) \\
&= [b_s \gamma_s b_{s-1} \ldots b_1 \gamma_1 (a_r.b_0) \beta_r a_{r-1} \ldots a_1 \beta_1 a_0] \\
&= [b_s \gamma_s b_{s-1} \ldots b_1 \gamma_1 (b_0._{op}a_r) \beta_r a_{r-1} \ldots a_1 \beta_1 a_0] \\
&= [b_s \gamma_s b_{s-1} \ldots b_1 \gamma_1 b_0][a_r \beta_r a_{r-1} \ldots a_1 \beta_1 a_0] \\
&= \phi([b_0 \gamma_1 b_1 \ldots b_{s-1} \gamma_s b_s])\phi([a_0 \beta_1 a_1 \ldots a_{r-1} \beta_r a_r]) = \phi(b)\phi(a)
\end{align*}
This proves that $k(\Gamma,\Aa)$ is anti-isomorphic to $k(\Gamma^{op},\Aa^{op})$ via $\phi$, which is the same to say that $k(\Gamma,\Aa)^{op}$ is isomorphic to $k(\Gamma^{op},\Aa^{op})$.
To conclude the proof, we realize that the map $\phi$ defined above satisfies $\phi(I) = I^{op}$, and the statement follows directly.
\end{proof}

\subsection{Duality} We now use the results of the previous subsection to dualize the representations of the gbp-algebra $\Lambda$. Denote by D = Hom$_k(-,k)$ the duality functor. %For a $\Lambda$-representation $M = ((M_i)_{i \in \Gamma_0},(\phi)_{\alpha \in \Gamma_1})$, we denote 

\begin{proposition}
\label{prop:dual of a repr}
Let $\Lambda = k(\Gamma,\Aa, I)$ be a gbp-algebra. If $((M_i)_{i \in \Gamma_0},(\phi_{\alpha})_{\alpha \in \Gamma_1})$ is the representation of the $\Lambda$-module $M$, then the representation of the  $\Lambda^{op}$-module D$M$ is isomorphic to 
$(D(M_i)_{i \in \Gamma_0}, D(\phi_{\alpha})_{\alpha \in \Gamma_1})$. 
\end{proposition}
\begin{proof} We need to show that the representations $(((DM)_i)_{i \in \Gamma_0},((DM)_{\alpha})_{\alpha \in \Gamma_1})$ and $(D(M_i)_{i \in \Gamma_0}, D(\phi_{\alpha})_{\alpha \in \Gamma_1})$ are isomorphic.
It is useful to recall how the quasi-inverse equivalences $F$ e $G$ discussed in the proof of Theorem~\ref{th:rep_icnlp} were like. Let $i \in \Gamma_0$. First of all, note that
$$DM = \Hom_k(M,k) \text{, thus } (DM)_i = 1_{i}(\Hom_k(M,k))$$
$$D(M_i) = \Hom_k(M_i,k) = \Hom_k(M.1_{i},k)$$
We can define
\begin{align*}
f_i: 1_{i} \Hom_k (M,k) & \rightarrow \Hom_k(M\cdot 1_{i},k) \\
1_{i}\cdot g & \mapsto g\vert_{M\cdot 1_{i}}
\end{align*}
We shall see that $f_i$ is an isomorphism. It is clear that it is well-defined and $k$-linear. To show that 
$f_i$ is a morphism of $A_i^{op}$-modules, let $g \in \Hom_k (M,k)$, $a \in A_i^{op}$ and $x \in M\cdot1_{i}$. Then 
$$f_i(a\cdot 1_{i} g)(x) = (a\cdot g)\vert_{M\cdot 1_{i}}(x) = (a\cdot g)(x) = g(xa) = g(xa1_{i}) = $$
$$ = g\vert_{M\cdot 1_{i}}(xa)  =  f_i(1_{i}g)(xa) = (a\cdot f_i(1_{i}g))(x) $$
which implies that $ f_i(a\cdot 1_{i}g) = a\cdot f_i(1_{i}g)$, as required. \\
Now, to see that $f_i$ is injective, suppose $f_i(1_{i}g)=0$. Then $(1_{i}g)(x) = 0$ for every $x \in M\cdot 1_{i}$ and so $(1_{i}\cdot g)(x) = (1_{i}\cdot g)(x\cdot 1_{i}) = 0$ for every $x \in M$. In particular, $1_{i}\cdot g = 0$, which shows our claim. \\
It remains to see that $f_i$ is surjective. Let $h \in \Hom_k(M\cdot 1_{i},k)$. We know that $M \cong \oplus_{j \in \Gamma_0} M\cdot 1_{j}$. We can thus define a $k$-linear transformation $g \in \Hom_k(M,k)$, $g: \oplus_{j \in \Gamma_0} M\cdot1_{j} \rightarrow k$, $g = (\delta_{ji} h)_{j \in \Gamma_0}$, where $\delta_{ji}$ is a Kronecker's delta. Then, if $x \in M\cdot 1_{i}$, $f_i(1_{i}\cdot g)(x) = g\vert_{M\cdot 1_{i}}(x) = h(x)$. Thus $f_i(1_{i}\cdot g) = h$.
This concludes the proof that $f_i$ is an isomorphism of $A_i$-modules. The next step is to show the commutativity of the diagram 
\begin{displaymath}
\xymatrix{
(DM)_j \ar[r]^{(DM)_{\alpha}} \ar[d]_{f_j} & (DM)_i \ar[d]^{f_i}\\
D(M_j) \ar[r]_{D(\phi_{\alpha})} & D(M_i)
}
\end{displaymath}
For that, let $g \in \Hom_k(M,k)$ e $x \in M$. Then:
\begin{align*}
(f_i \circ (DM)_{\alpha})(1_{j}.g)&(x.1_{i}) =f_i( (DM)_{\alpha}(1_{j}.g))(x.1_{i}) = f_i(1_{i} \alpha g)(x.1_{i}) \\
&= (\alpha g)\vert_{M.1_{i}} (x.1_{i})= (\alpha g) (x.1_{i}) =g(x \alpha) \\
&=g\vert_{M.1_{j}}(x \alpha 1_{j}) = g\vert_{M.1_{j}}(x \alpha) = g\vert_{M.1_{j}}(\phi_{\alpha}(x.1_{i})) \\
&= D(\phi_{\alpha})(g\vert_{M.1_{j}})(x.1_{i}) = D(\phi_{\alpha})(f_j(1_{i}.g))(x.1_{i})) \\
&= (D(\phi_{\alpha}) \circ f_j)(1_{i}.g)(x.1_{i})
\end{align*}
Hence $(f_i \circ (DM)_{\alpha}) = (D(\phi_{\alpha}) \circ f_j)$, as was required. The fact that $DM$ satisfies $I^{op}$ if and only if $M$ satisfies $I$ follows easily from the fact that D is a fully faithful and dense $k$-linear functor.
\end{proof}

%%%%%%%%%%%%%%%%%%%%%%%%%%%%

\section{Realizing an $A_i$-module as a $\Lambda$-module}
\label{sec:realizing modules}

Let $i \in \Gamma_0$, and let $M$ be a (right) $A_i$-module. In this section we shall see three ways of viewing $M$ as a $\Lambda$-module. The first one is quite natural, while the second one essentially relies on the well-known technique of extension of scalars. Dualizing such a construction, we get a third way. It will be interesting to dedicate different notations for each of the three.

\subsection{The inclusion functors} \label{subsec:inclusion} 
Given an $A_i$-module $M$, define the  $\Lambda$-representation 
${\mathcal{I}}(M) = ((M_j)_{j \in \Gamma_0},(\phi_{\alpha})_{\alpha \in \Gamma_1}) $ given by
$$M_j = 
\begin{cases}
M & \text{ if } j = i \\
0 & \text{ if } j \neq i
\end{cases} \ \ \ \mbox{ and } \ \ \ \phi_{\alpha} = 0 \hspace{1ex}\text{ for all }\alpha \in \Gamma_1. $$

Clearly, because of Theorem~\ref{th:rep_icnlp}, $\I(M)$ yields a $\Lambda$-module, and, since $\I(M)$ and $M$ have the same underlying vector space, we may, by abuse of notation, denote $\I(M) = M$.

Actually, for every vertex $i$ we have a functor $\I_i: \md A_i \rightarrow \md \Lambda$ which we shall call \textbf{inclusion functor}. (We might even denote it simply by $\I$ if it is clear what vertex we are talking about). We have just defined its image on objects, and its image on morphisms is defined obviously. It is also easy to see why $\I$ is called an inclusion functor, because it is covariant and fully faithful.

From now on, unless stated or denoted otherwise, we will always be assuming that we are seeing $M$ as an $\Lambda$-module in this way.

\begin{remark}
\label{obs:simple modules}
It is not difficult to see that simple $A_i$-modules viewed as $\Lambda$-modules are also simple. Conversely, any simple $\Lambda$-module is of this kind. This follows from a counting argument (see \cite{CLiu} for example). So, the description of the simple $\Lambda$-modules is easily done.
\end{remark}

\subsection{Cones} We shall now see another way to view an $A_i$-module $M$ as a $\Lambda$-module. 

Here again, let $k(\Gamma,\Aa) = T(A_{\Aa},M_{\Aa})$ as in Remark~\ref{obs:gpa_tensor_algebra}. Clearly, $M$ is also an $A_{\Aa}$-module (using the action $m \cdot (a_j)_j = m\cdot a_i$ for each $m \in M$ and $(a_j)_{j \in \Gamma_0} \in A_{\Aa}$). 

Since $\Lambda$ is equal to the quotient $k(\Gamma,\Aa)/I$,  and $M_{\Aa}$ is an $(A_{\Aa}-A_{\Aa})$-bimodule, $\Lambda$  is also an $(A_{\Aa}-A_{\Aa})$-bimodule that contains $A_{\Aa}$ as a subalgebra. Therefore it makes sense to consider the extension of scalars of $M$ to $\Lambda$. We shall denote it by $\Co_i(M) = M \otimes_{A_{\Aa}} \Lambda$. Just emphasizing, since $\Lambda$ is a right $\Lambda$-module, $\Co_i(M)$ is a right $\Lambda$-module too. 
\begin{definition}
$\Co_i(M)$ is called \textbf{cone} \index{cone} over $M$.
\end{definition}
The reason why we call it a cone is because of the shape that the representation of $\Co_i(M)$ has, as it will be more transparent after the description that will be done here later. 
\begin{proposition}
\label{cone x direct_sum}
If $M$ and $N$ are $A_i$-modules, then $\Co_i(M \oplus N) \cong \Co_i(M) \oplus \Co_i(N)$.
\end{proposition}
\begin{proof} Just observe that 
$$\Co_i(M \oplus N) = (M \oplus N)\otimes_{A_{\Aa}} \Lambda \cong (M \otimes_{A_{\Aa}} \Lambda) \oplus (N \otimes_{A_{\Aa}} \Lambda ) = \Co_i(M) \oplus \Co_i(N).$$
\end{proof}

%%% converse of the above proposition. 

\begin{remark}
\label{obs:cone as formal compositions}
Since we are assuming $\Gamma$ to be acyclic, it will be useful to remark that
$$\Co_i(M) = \left\{ \sum_{\substack{\gamma = \gamma_1 \ldots \gamma_t \text{ is a path in } \Gamma \\ s(\gamma_1) = i}} m^{\gamma} \otimes \overline{\gamma_1 a^{\gamma}_{e(\gamma_1)} \ldots \gamma_t a^{\gamma}_{e(\gamma_t)}} : m^{\gamma} \in M, a^{\gamma}_{e(\gamma_j)} \in A_{e(\gamma_j)} \right\}$$

This equality follows by observing that $\Co_i(M) = M \otimes_{A_{\Aa}} \Lambda = M\cdot 1_i \otimes_{A_{\Aa}} \Lambda  = M \otimes_{A_{\Aa}} 1_i\cdot\Lambda$.
\end{remark}

The next goal of this subsection is to describe the representation associated to the cone $\Co_i(M)$ of an $A_i$-module $M$. 

Let $((M_j)_{j \in \Gamma_0},(\phi_{\alpha})_{\alpha \in \Gamma_1})$ denote the representation of $M$. For each $l \in \Gamma_0$, let $\{a_1^l,\ldots,a_{\dimk A_l}^l\}$ denote a $k$-basis of $A_l$. Also, let $\{m_1,\dots,m_{\dimk M}\}$ be a $k$-basis of $M$. 
\begin{proposition}
\label{prop:formato_dos_cones}
With the notations above, it holds that $M_i = M$, and if $j \in \Gamma_0$ is different from $i$, then $M_j$ is isomorphic to the free $A_j$-module having as basis the set of equivalence classes of the formal sequences of the form
$$m_p  \gamma_1 a_{i_2}^{s(\gamma_2)} \ldots a_{i_r}^{s(\gamma_r)} \gamma_r$$
where $\gamma_1 \ldots \gamma_r$ is a path from $i$ to $j$, $1 \leq p \leq \dimk M$ and $1 \leq i_l \leq \dimk A_{s(\gamma_l)}$ for every $1<l \leq r$.

Moreover, if $\alpha:j \rightarrow j'$ is an arrow, then $\phi_{\alpha}$ is the only linear transformation that satisfies
$$\phi_{\alpha}\left(\overline{m_p  \gamma_1 a_{i_2}^{s(\gamma_2)} \ldots a_{i_r}^{s(\gamma_r)} \gamma_r} a_{i_{r+1}}^j\right) = \overline{m_p  \gamma_1 a_{i_2}^{s(\gamma_2)} \ldots a_{i_r}^{s(\gamma_r)} \gamma_r a_{i_{r+1}}^j \alpha}.$$
\end{proposition}
\begin{proof}
The key idea here is to recall the equivalence $G$ constructed in the proof of Theorem~\ref{th:rep_icnlp}.
By Remark~\ref{obs:cone as formal compositions} above, and by the fact that $\Gamma$ is acyclic,
$$M_i = \Co_i(M)\cdot 1_i \cong \left\{ \sum_{\gamma:i \rightsquigarrow i} m^{\gamma}: m^{\gamma} \in M \right\} = \{m: m \in M\} =M$$
For $j \neq i$, we have that
\begin{align*}
M_j = \Co_i(M)\cdot 1_j &= \left\{  \sum_{\gamma=\gamma_1 \ldots \gamma_{r}:i \rightsquigarrow j} m^{\gamma} \otimes \overline{\gamma_1 a_2^{\gamma} \gamma_2 \ldots a_r^{\gamma} \gamma_r a_{r+1}^{\gamma}} : \right. \\
 & \left. \hspace{1cm} m^{\gamma} \in M, a_l^{\gamma} \in A_{s(\gamma_l)} \hspace{1ex} \forall 1< l \leq r, \text { e } a_{r+1}^{\gamma} \in A_j \right\}
\end{align*}
Since $\{a_1^l,\ldots,a_{\dimk A_l}^l\}$ is a $k$-basis of $A_l$ and $\{m_1,\dots,m_{\dimk M}\}$ is a $k$-basis of $M$, the above expression equals to
\begin{align}
&\operatorname{span}_k \{m_p \otimes \overline{ \gamma_1 a_{i_2}^{s(\gamma_2)} \ldots a_{i_r}^{s(\gamma_r)} \gamma_r} a_{r+1}: \gamma_1 \ldots \gamma_r \text{ is a path } i \rightsquigarrow j, \nonumber \\
& \hspace{7mm} 1 \leq p \leq \dimk M, 1 \leq i_l \leq \dimk A_{s(\gamma_l)} \hspace{1ex} \forall 1< l \leq r, \text { e } a_{r+1} \in A_j\} \label{eq:prop}
\end{align}
If one denotes $\{\theta_1,\ldots,\theta_{n_j}\} = \{m_p \otimes \overline{ \gamma_1 a_{i_2}^{s(\gamma_2)} \ldots a_{i_r}^{s(\gamma_r)} \gamma_r}\}$, then  the expression~\ref{eq:prop} is equal to 
$$\operatorname{span}_k \{ \theta_l a: 1 \leq l \leq n_j, a \in A_j \}.$$
An easy calculation shows that it is isomorphic to the free $A_j$-module having as basis $\{\theta_1,\ldots,\theta_{n_j}\}$, as we wanted to prove. \\
Let $\alpha: j \rightarrow j'$ be an arrow in $\Gamma_1$. Again, by Theorem~\ref{th:rep_icnlp}, $\phi_{\alpha}:  M_j \rightarrow M_{j'}$ is given by
\begin{align*}
\phi_{\alpha}:  \Co_i(&M) 1_j \rightarrow \Co_i(M) 1_{j'} \\
& m 1_j \mapsto m \alpha
\end{align*}
with $m \in \Co_i(M)$. Therefore $\phi_{\alpha}$ has the form given in the statement, concluding the proof.
\end{proof}

\begin{remark}
\label{obs:formato_dos_cones}
If $I = 0$, then it is easier to see how the representation of $\Co_i(M)$ looks like: it holds that $M_i = M$, and if $j \neq i$, $M_j \cong A_j^{n_j}$, where 
$$n_j = \sum_{\gamma: i = i_0 \rightarrow i_1 \rightarrow \ldots \rightarrow i_{r+1} = j \text{ is a path }i \rightsquigarrow j} (\dimk M).(\dimk A_{i_1}).\ldots.(\dimk A_{i_r}) $$
In particular, if there is no path going from $i$ to $j$, $M_j = 0$.
\end{remark}

We finish this subsection with the following result. 
\begin{proposition}
\label{prop:cone is exact}
Given $i \in \Gamma_0$, the cone functor $\Co_i : \md A_i \rightarrow \md \Lambda$  is exact.
\end{proposition}
\begin{proof}
By definition, $\Co_i \doteq I_i(-) \otimes_{A_{\Aa}} \Lambda$. Since the inclusion functor $I_i: \md A_i \rightarrow \md \Lambda$ is easily seen to be exact and a tensor product $- \otimes_{A_{\Aa}} \Lambda$ is always right exact, $\Co_i$ is right exact. Our work here is to prove that $\Co_i$ maps monomorphisms to monomorphisms, because then $\Co_i$ will also be left exact and thus exact, concluding the proof. So let $f: M \rightarrow N$ be a monomorphism between $A_i$-modules. Then it is sufficient to fix $j \in \Gamma_0$ and prove that $(\Co_i(f))_j:(\Co_i(M))_j \rightarrow (\Co_i(N))_j$ is a monomorphism of $A_j$-modules. \\ 
If there is no path $i \rightsquigarrow j$ in $\Gamma$, then we know that $(\Co_i(f))_j$ will be a zero map between two zero modules and thus a monomorphism. So we may suppose that there are paths of the form $i \rightsquigarrow j$ in $\Gamma$. \\
Then, if $\{m_1,\ldots,m_r\}$ is a $k$-basis of $M$, the set $\{f(m_1),\ldots,f(m_r)\} \subset N$ will be linearly independent. Therefore, if we denote $f(m_l) = n_l$ for every $l$, we can complete this set to a $k$-basis of $N$: $\{n_1,\ldots,n_r,\ldots,n_s\}$. Also, for every vertex $l$, let $\{a_1^l,a_2^l\ldots,a_{n_l}^l\}$ be a $k$-basis of $A_j$. \\ 
Let $\gamma: i = l_0 \rightarrow l_1 \rightarrow \ldots \rightarrow l_t = j$ be a path between $i$ and $j$ in $\Gamma$. Then we denote
$$\theta_{\gamma,h,i_1,\ldots,i_{t-1}} = m_h \otimes \overline{\gamma_1 a_{i_1}^{l_1} \gamma_2 \ldots \gamma_{t-1} a_{i_{t-1}}^{l_{t-1}}} \in \Co_i(M)$$
$$\zeta_{\gamma,h,i_1,\ldots,i_{t-1}} = n_h \otimes \overline{\gamma_1 a_{i_1}^{l_1} \gamma_2 \ldots \gamma_{t-1} a_{i_{t-1}}^{l_{t-1}}} \in \Co_i(N)$$
And we note that 
\begin{align*}
\Co_i(f)(\theta_{\gamma,h,i_1,\ldots,i_{t-1}}) &= (f \otimes 1_{\Lambda})(m_h \otimes \overline{\gamma_1 a_{i_1}^{l_1} \gamma_2 \ldots \gamma_{t-1} a_{i_{t-1}}^{l_{t-1}}}) \\
&= f(m_h) \otimes \overline{\gamma_1 a_{i_1}^{l_1} \gamma_2 \ldots \gamma_{t-1} a_{i_{t-1}}^{l_{t-1}}} \\
&= n_h \otimes \overline{\gamma_1 a_{i_1}^{l_1} \gamma_2 \ldots \gamma_{t-1} a_{i_{t-1}}^{l_{t-1}}} = \zeta_{\gamma,h,i_1,\ldots,i_{t-1}}
\end{align*}
By Proposition~\ref{prop:formato_dos_cones}, we know that $(\Co_i(M))_j$ is the free $A_j$-module generated by the $\theta_{\gamma,h,i_1,\ldots,i_{t-1}}$, while $(\Co_i(N))_j$ is the free $A_j$-module generated by the $\zeta_{\gamma,h,i_1,\ldots,i_{t-1}}$. So $(\Co_i(f))_j$, which is a restriction of $\Co_i(f)$, is a morphism that takes a  basis of $(\Co_i(M))_j$ to a subset of a  basis of $(\Co_i(N))_j$. Therefore, it must be a monomorphism, concluding the proof.
\end{proof}

\subsection{Dual cones} 

We now dualize the notion of cone. 

\begin{definition}
Let $i \in \Gamma_0$, and let $M$ be an $A_i$-module. Then $D(M)$ is an $A_i^{op}$-module, and therefore the cone $\Co_i(DM)$ is a $\Lambda^{op}$-module. Finally, $D(\Co_i(DM))$ is a $\Lambda$-module, which we call \textbf{dual cone} of $M$. We shall use the notation $\dCo_i(M) \doteq D(\Co_i(DM))$.
\end{definition}
\begin{proposition}
\label{prop:dual cones x sum}
Given two $A_i$-modules $M$ and $N$, $\dCo_i(M \oplus N) \cong \dCo_i(M) \oplus \dCo_i(N)$.
\end{proposition}
\begin{proof}
This follows because the duality functor preserves direct sums and because $\Co_i$ also preserves direct sums due to Proposition~\ref{cone x direct_sum}.
\end{proof}

\begin{example}
Let us give an example to illustrate the differences between the three ways of realizing an $A_i$-module as a $\Lambda$-module seen in this section. \\
Let $A$ and $B$ be two finite dimensional algebras over the base field $k$. Suppose that $A$ has dimension 2 over $k$ and that $B$ has dimension 3. Consider the gbp-algebra $\Lambda$ given below:
\begin{displaymath}
\xymatrix{
&&B \ar[r]&A \\
B \ar[r]&k \ar[ur] \ar[r]_{\alpha} &A \ar[r]_{\beta}&B }
\end{displaymath}
bound by $\alpha \beta = 0$. Let $x$ be the vertex of the quiver above to which $k$ was assigned. If we consider $k^4$ as a $\Lambda$-module via the inclusion functor relative to $x$, its representation will be

\begin{displaymath}
\xymatrix{
&&0 \ar[r]&0 \\
0 \ar[r]&k^4 \ar[ur] \ar[r] &0 \ar[r]&0 }
\end{displaymath}
By using Proposition~\ref{prop:formato_dos_cones} above, one concludes that the representation of $\Co_{x}(k^4)$, which is the cone of $k^4$, will be

\begin{displaymath}
\xymatrix{
&&B^4 \ar[r]& A^{12} \\
0 \ar[r]&k^4 \ar[ur] \ar[r] &A^4 \ar[r]&0 }
\end{displaymath}
The bottom right vertex needs to be assigned with 0 as a consequence of the existence of the relation $\alpha \beta =0$. Note how the representation of $\Co_{x}(k^4)$ resembles a cone whose vertex is $x$ and whose basis is the set of vertices which are the end of non-zero paths starting at $x$. This is to complement our previous remark explaining why we are calling the functor $\Co_x$ a cone. Finally, the dual cone $\dCo_x(k^4)$ of $k^4$ will be given by 

\begin{displaymath}
\xymatrix{
&&0 \ar[r]& 0 \\
B^4 \ar[r]&k^4 \ar[ur] \ar[r] &0 \ar[r]&0 }
\end{displaymath}
\end{example}

\begin{remark} 
\label{obs:calculating dual cones}
We gave above a description of the representations associated with cones. That is, we already know how to calculate cones. Thanks to Proposition~\ref{prop:dual of a repr}, calculating dual cones will not present a difficulty any bigger: given an $A_i$-module $M$, we calculate the cone of $DM$ over $(\Gamma^{op},\Aa^{op},I^{op})$ and then obtain the dual cone of $M$ over $(\Gamma,\Aa,I)$ using Proposition~\ref{prop:dual of a repr}. This propositions tells us that what we need to do is to take the duals of the modules in each vertex and take the transpose linear transformation in each arrow, which, in practical situations, is done by transposing matrices. We shall yield examples of this in Subsection~\ref{subsec:injectives}.
\end{remark}

\section{Projective and injective representations}

We shall now apply the results of the previous subsection to describe the indecomposable projective  and injective $\Lambda$-modules. We remark that \cite{LY} contains a description of projective modules over generalized path algebras, although here we manage to extend this to the context of gbp-algebras.

\subsection{Projective representations}
\label{subsec:projective}
We start with the following result. 

\begin{proposition}
\label{cone x projective}
If $P$ is a projective $A_i$-module, then $\Co_i(P)$ is a projective $\Lambda$-module.
\end{proposition}
\begin{proof}
Let $g: M \rightarrow N$ be an epimorphism of $\Lambda$-modules. Since $\Lambda$ is a projective $\Lambda$-module,
$$\Hom_{\Lambda}(\Lambda,g):\Hom_{\Lambda}(\Lambda,M) \rightarrow \Hom_{\Lambda}(\Lambda,N)$$
is an epimorphism. Since $A_i = 1_i A_{\Aa}$ and $1_i$ is an idempotent element of $A_{\Aa}$, $A_i$ is a projective $A_{\Aa}$-module. By hypothesis, $P$ is a direct summand of some $A_i$-module of the form $A_i^m$, with $m \in \mathbb{N}$, and thus also $P$ is projective as an $A_{\Aa}$-module. It follows that
$$\Hom_{A_{\Aa}}(P,\Hom_{\Lambda}(\Lambda,g)):\Hom_{A_{\Aa}}(P,\Hom_{\Lambda}(\Lambda,M)) \rightarrow \Hom_{A_{\Aa}}(P,\Hom_{\Lambda}(\Lambda,N))$$
is an epimorphism. Finally, by the Adjunction Theorem,
$$\Hom_{\Lambda}(P \otimes_{A_{\Aa}} \Lambda,g) : \Hom_{\Lambda}(P \otimes_{A_{\Aa}} \Lambda,M) \rightarrow \Hom_{\Lambda}(P \otimes_{A_{\Aa}} \Lambda,N)$$
is an epimorphism. This proves that $P \otimes_{A_{\Aa}} \Lambda$ is a projective $\Lambda$-module.
\end{proof}

Now, for each $i \in \Gamma_0$, let $E_i = \{e_{i1},\ldots,e_{is_i}\}$ be a complete set of primitive idempotent and pairwise orthogonal elements in $A_i$. Then every indecomposable projective $A_i$-module is isomorphic to $P_i^j \doteq e_{ij}A_i$ for some $1 \leq j \leq s_i$. Moreover, $E = \{\overline{e_{ij}}:i \in \Gamma_0,1 \leq j \leq s_i\}$ is a complete set of primitive idempotent and pairwise orthogonal elements in $\Lambda$. Therefore every indecomposable projective $\Lambda$-module is isomorphic to $P(i,j) \doteq \overline{e_{ij}} \Lambda$ for a certain pair of indexes $i \in \Gamma_0$ e $1 \leq j \leq s_i$.

\begin{proposition}
\label{prop:description_proj}
For each $i \in \Gamma_0$ and $1 \leq j \leq s_i$, $P(i,j) = \Co_i(P_i^j)$.
\end{proposition}
\begin{proof}
Using Remark~\ref{obs:cone as formal compositions}, we have that 

\begin{align*}
\Co_i(P_i^j) &= \left\{ \sum_{\substack{\gamma = \gamma_1 \ldots \gamma_t \text{ in } \Gamma \\ s(\gamma_1) = i}} m^{\gamma} \otimes \overline{ \gamma_1 a^{\gamma}_{e(\gamma_1)} \ldots \gamma_t a^{\gamma}_{e(\gamma_t)}} : m^{\gamma} \in P_i^j, a^{\gamma}_{e(\gamma_j)} \in A_{e(\gamma_j)} \right\} \\
&=\left\{ \sum_{\substack{\gamma = \gamma_1 \ldots \gamma_t \text{ in } \Gamma \\ s(\gamma_1) = i}} \overline{e_{ij} a^{\gamma} \gamma_1 a^{\gamma}_{e(\gamma_1)} \ldots \gamma_t a^{\gamma}_{e(\gamma_t)}} : a^{\gamma} \in A_i, a^{\gamma}_{e(\gamma_j)} \in A_{e(\gamma_j)} \right\} \\
&= \overline{e_{ij}}\Lambda = P(i,j)
\end{align*}
\end{proof}

Thanks to the last proposition and Proposition~\ref{prop:formato_dos_cones}, we are now able to calculate the representations associated to projective indecomposable modules. 
The following proposition reflects the particular case of this construction when $I=0$, i.e., when there are no relations:

\begin{proposition}
\label{prop:proj with no relations}
Suppose $I=0$. Let $P(i,j) = ((M_i)_{i \in \Gamma_0},(\phi_{\alpha})_{\alpha \in \Gamma_0})$ be the representation associated to $P(i,j)$. Then, for $l \in \Gamma_0$,
\benu
\item[(a)] If $l = i$, then $M_l = M_i = P_i^j$.
\item[(b)] If $l \neq i$, denote
$$n_l = \sum_{\gamma: i = i_0 \rightarrow i_1 \rightarrow \ldots \rightarrow i_r = l} (\dimk P_i^j)\cdot (\dimk A_{i_1})\cdot \ldots.(\dimk A_{i_{r-1}}) $$
where $\gamma$ runs through all possible paths $i \rightsquigarrow l$. 
\enu
Then $M_l \cong (A_l)^{n_l}$ as $A_l$-modules. In particular, if there are no paths $i \rightsquigarrow l$, then $M_l = 0$.
\end{proposition}

In practical examples, however, difficulties may arise either because the matrices of the $k$-linear transformations denoted above as $\phi_{\alpha}$ can be too big, or, given their dependence on the choice of a $k$-basis of the algebras $A_i$ or of $P_i^j$, there could be some confusion. To avoid that, it is convenient to make use of block matrices. We shall give further details of this in the remark and example below.

\begin{remark}
Let $V$ be a $k$-vector space of dimension 1 and fixed basis $\{v\}$ and let $A$ be a $k$-algebra. Then there is a linear map that shall be treated as canonical from now on: it is defined as $\mu : V \rightarrow A$, $\mu(\lambda\cdot v) = \lambda\cdot 1_A$, where $\lambda \in k$. Although the vector space $V$ may vary, the letter $\mu$ will always be used for such a map.
\end{remark}

\begin{example}
\label{ex:projectives_1}
Let $A$ be the path algebra given by the quiver below:
\begin{displaymath}
\xymatrix{
1 \ar[rr] & & 2
}
\end{displaymath}
Then there are two indecomposable projective $A$-modules, namely,
\begin{displaymath}
\xymatrix{
P_1: \hspace{1ex} k \ar[rr]^{id} & & k & &P_2: \hspace{1ex} 0 \ar[rr] & & k
}
\end{displaymath}
Now let $\Lambda$ be the generalized path algebra given by
\begin{displaymath}
\xymatrix{
A \ar[rr] & & A
}
\end{displaymath}

According to the discussions above, there are exactly 4 indecomposable projective $\Lambda$-modules, which are:

\begin{displaymath}
\xymatrix{
P(1,1): \hspace{1ex} P_1 \ar[rr]^(0.6){\left[ \begin{array}{cc}
\mu & 0 \\ 
0 & \mu
\end{array} \right]} & & A^2 & & P(1,2): \hspace{1ex} P_2 \ar[rr]^(0.6){\left[ \begin{array}{c}
\mu 
\end{array} \right]} & & A \\
P(2,1): \hspace{1ex} 0 \ar[rr] & & P_1 & &P(2,2): \hspace{1ex} 0 \ar[rr] & & P_2
}
\end{displaymath}

\end{example}

We also have conditions to describe the representations associated to radicals of the projective modules, as expressed in the proposition below:

\begin{proposition}
\label{prop:radical of a projective}
With the same notations as before, let $i \in \Gamma_0$ and $1 \leq j \leq s_i$. Denote $P(i,j) = ((M_l)_{l \in \Gamma_0},(\phi_{\alpha})_{\alpha \in \Gamma_1})$. Then the radical of $P(i,j)$ is given by the representation $\rad P(i,j) = ((N_l)_{l \in \Gamma_0},(\psi_{\alpha})_{\alpha \in \Gamma_1})$, where $N_i = \rad P_i^j$, $N_l = M_l$ for each $l \in \Gamma_0$ with $l \neq i$, and for each $\alpha \in \Gamma_1$, $\psi_{\alpha} = \phi_{\alpha}\vert_{M_{s(\alpha)}}$.
\end{proposition}

\begin{proof}
Let $N = ((N_l)_{l \in \Gamma_0},(\psi_{\alpha})_{\alpha \in \Gamma_1})$. Note that $N$ satisfies $I$ because $M$ satisfies it. We wish to prove that $N = \rad P(i,j)$.
Note that, if $l \neq i$, $N_l = M_l$, so $M_l / N_l = 0$. Moreover, $M_i = P_i^j$ and $N_i = \rad P_i^j$, thus $M_i/N_i = P_i^j / \rad P_i^j$. This implies that $P(i,j)/N$ is isomorphic to the $A_i$-module $P_i^j / \rad P_i^j$ realized as a $\Lambda$-module. Since $P_i^j$ is an indecomposable projective $A_i$-module, $P_i^j / \rad P_i^j$ is a simple $A_i$-module, and it is also simple when seen as a $\Lambda$-module, according to Remark~\ref{obs:simple modules}. This means that $P(i,j)/N$ is a simple $\Lambda$-module.  We have thus proved that $N$ is a maximal submodule of $P(i,j)$, and since $P(i,j)$ is indecomposable projective, it has a unique maximal submodule, which is $\rad P(i,j)$. This concludes the proof that $N = \rad P(i,j)$.
\end{proof}

\begin{example}
We continue Example~\ref{ex:projectives_1} above to apply Proposition~\ref{prop:radical of a projective} and thus obtain the radical of the 4 projective modules seen above. Thus we have:

\begin{displaymath}
\xymatrix{
\rad P(1,1): \hspace{1ex} \rad P_1 \ar[rr]^(0.6){\left[ \begin{array}{c}
0 \\ 
\mu
\end{array} \right]} & & A^2 &  \rad P(1,2): \hspace{1ex} 0 \ar[rr] & & A \\
\rad P(2,1): \hspace{1ex} 0 \ar[rr] & & \rad P_1 &  \rad P(2,2): \hspace{1ex} 0 \ar[rr] & & 0
}
\end{displaymath}

\end{example}

\subsection{Injective representations}
\label{subsec:injectives}

In this subsection we shall give a description of the representations associated with indecomposable injective modules. As we shall see, the injective modules will be particular cases of dual cones, in an analogy with the projective modules, which were particular cases of cones, as we saw in Subsection~\ref{subsec:projective}.

\begin{proposition}
For $i \in \Gamma_0$, if $I$ is an injective $A_i$-module, then $\dCo_i (I)$ is an injective $\Lambda$-module.
\end{proposition}
\begin{proof}
Since $I$ is an injective $A_i$-module and $D$ is a duality, $DI$ is a projective $A_i^{op}$-module. Because of Proposition~\ref{cone x projective}, $\Co_i(DI)$ is a projective $\Lambda^{op}$-module, and again since $D$ is a duality, $\dCo_i(I) = D(\Co_i(DI))$ is an injective $\Lambda$-module.
\end{proof}

For each $i \in \Gamma_0$, let $E_i = \{e_{i1},\ldots,e_{is_i}\}$ be a complete set of primitive idempotent and pairwise orthogonal elements in $A_i$. If $D: \md A_i^{op} \rightarrow \md A_i$ is the duality functor, then its well-known that a  complete set of isomorphism classes of indecomposable injective $A_i$-modules is given by $I_i^1 = D(A_i e_{i1}),\ldots,I_i^{s_i} = D(A_i e_{is_i})$.

On the other hand, if $E = \{\overline{e_{ij}}:i \in \Gamma_0, 1 \leq j \leq s_i\}$, then $E$ is a complete set of primitive idempotent and pairwise orthogonal elements in $\Lambda$. This means that a complete set of isomorphism classes of indecomposable injective $\Lambda$-modules is given by $\{I(i,j): i \in \Gamma_0, 1 \leq j \leq s_i \}$, where $I(i,j) \doteq D(\Lambda \overline{e_{ij}})$.

\begin{proposition}
\label{prop:description_inj}
With the notations above, $\dCo_i(I_i^j) \cong I(i,j)$.
\end{proposition}
\begin{proof} 
$$\dCo_i(I_i^j) = D(\Co_i(D(I_i^j))) = D(\Co_i(D(D(A_i e_{ij})))) \cong D(\Co_i(A_i e_{ij})) = D(\Lambda \overline{e_{ij}}) = I(i,j)$$
where the penultimate equality follows from Proposition~\ref{prop:description_proj}.
\end{proof}

Proposition~\ref{prop:description_inj} gives us a complete description of the indecomposable injective $\Lambda$-modules. In order to calculate these modules in practical examples, we need to combine this description with Remark~\ref{obs:calculating dual cones} above. 

The particular case of when there are no relations is expressed in the following proposition, which is dual to Proposition~\ref{prop:proj with no relations} above:

\begin{proposition}
Suppose $I=0$. Let $I(i,j) = ((M_i)_{i \in \Gamma_0},(\phi_{\alpha})_{\alpha \in \Gamma_0})$ be the representation associated to $I(i,j)$. Then, for $l \in \Gamma_0$,
\benu
\item[(a)] If $l = i$, then $M_l = M_i = I_i^j$.
\item[(b)] If $l \neq i$, denote
$$n_l = \sum_{\gamma: l = i_0 \rightarrow i_1 \rightarrow \ldots \rightarrow i_r = i} (\dimk A_{i_1}).\ldots.(\dimk A_{i_{r-1}})(\dimk I_i^j). $$
where $\gamma$ runs through all possible paths $l \rightsquigarrow i$. Then $M_l \cong (A_l^*)^{n_l}$ as $A_l$-modules, where we denote $A_l^* = D(A_l)$ for brevity. In particular, if there are no paths $l \rightsquigarrow i$, then $M_l = 0$.
\enu
\end{proposition}

\begin{example}

Let $A$ be the path algebra given by the quiver
\begin{displaymath}
\xymatrix{
1 & & 2 \ar[ll] 
}
\end{displaymath}
Then there are 2 indecomposable injective $A$-modules, namely,
\begin{displaymath}
\xymatrix{
I_1: \hspace{1ex} k & & k\ar[ll]_{id} & &I_2: \hspace{1ex} 0 & & k\ar[ll]
}
\end{displaymath}
Now let $\Lambda$ be the generalized path algebra given by
\begin{displaymath}
\xymatrix{
A & & A \ar[ll]
}
\end{displaymath}

We want to calculate the indecomposable injective $\Lambda$-modules. According to the discussions above, we first calculate the indecomposable projective modules over the following generalized path algebra:
\begin{displaymath}
\xymatrix{
A^{op} \ar[rr] & & A^{op}
}
\end{displaymath}

and we note that $A^{op}$ is the path algebra over the following quiver:
\begin{displaymath}
\xymatrix{
1 \ar[rr] & & 2 
}
\end{displaymath} 

In our case, this calculation was already done in Example~\ref{ex:projectives_1}. Therefore it remains only to apply Proposition~\ref{prop:dual of a repr}. Thus the indecomposable injective $\Lambda$-modules are:

\begin{displaymath}
\xymatrix{
I(1,1): \hspace{1ex} I_1 & & & (A^*)^2 \ar[lll]_(0.4){\left[ \begin{array}{cc}
D(\mu) & 0 \\ 
0 & D(\mu)
\end{array} \right]}  & I(1,2): \hspace{1ex} I_2 & & A^* \ar[ll]_(0.4){\left[ \begin{array}{c}
D(\mu)
\end{array} \right]}  \\
I(2,1): \hspace{1ex} 0 & & I_1 \ar[ll] &  & I(2,2): \hspace{1ex} 0 & & I_2 \ar[ll]
}
\end{displaymath}

\end{example}

\backmatter

\bmhead{Acknowledgments}

 The authors gratefully acknowledge financial support by São Paulo Research Foundation (FAPESP), grants \#2018/18123-5, \#2020/13925-6 and  \#2022/02403-4. The second author has also a grant by CNPq (Pq 312590/2020-2).
 
 \section*{Statements and Declarations}

Competing interests: The authors declare they do not have potential interests apart from funding specified above.

%Availability of data and materials: Data sharing is not applicable to this article as no datasets were generated or analysed during its production.

\end{document}